\documentclass[a4paper,11pt,titlepage,twoside]{article}

\usepackage{graphicx}
\usepackage[T1]{fontenc} 
\usepackage[utf8]{inputenc}
\usepackage[english]{babel}
\usepackage{soul}
\usepackage{amsfonts}
\usepackage{amsmath}
\usepackage{amsthm}
\usepackage{amssymb}
\usepackage{mathrsfs}
\usepackage[top=5cm, bottom=3cm, left=3.3cm, right=3.3cm]{geometry}
\usepackage{setspace}
\usepackage{afterpage}
\usepackage{extarrows}
\usepackage{fancyhdr}
\usepackage{titlesec}
\usepackage{enumitem} \setlist{nosep}
\usepackage[pdftex,breaklinks,colorlinks,linkcolor=blue,
anchorcolor=blue]{hyperref}

\theoremstyle{definition}
\newtheorem{defin}{Definition}[section]
\theoremstyle{plain}
\newtheorem{theo}[defin]{Theorem}
\newtheorem{lem}[defin]{Lemma}
\newtheorem{pro}[defin]{Proposition}
\newtheorem{cor}[defin]{Corollary}
\theoremstyle{definition}
\newtheorem{exm}[defin]{Example}
\newtheorem{rem}[defin]{Remark}

\renewcommand{\H}{\mathcal{H}}

\newcommand{\rn}{\mathfrak{n}}

\newcommand{\n}[1]{\|#1\|}
\newcommand{\nor}{\|\cdot\|}

\renewcommand{\l}{\langle}
\renewcommand{\r}{\rangle}
\newcommand{\N}{\mathbb{N}}
\newcommand{\Z}{\mathbb{Z}}

\newcommand{\R}{\mathbb{R}}
\newcommand{\C}{\mathbb{C}}

\newcommand{\pint}{\l\cdot,\cdot\r}
\newcommand{\pin}[2]{\l#1 , #2\r}
\newcommand{\no}{\noindent}

\newcommand{\mez}{\frac{1}{2}}
\renewcommand{\phi}{\varphi}
\numberwithin{equation}{section}

\renewcommand\labelenumi{\emph{(\roman{enumi})}}
\renewcommand\theenumi\labelenumi

\titleformat{\section}
{\normalfont\fillast \fontsize{12}{15}\scshape}{\thesection.}{0.8em}{}

\titleformat{\subsection}
{\normalfont\fillast \fontsize{11}{12}\scshape}{\thesubsection.}{0.8em}{}

\begin{document}
	
\thispagestyle{plain}

\begin{center}
	\large
	{\uppercase{\bf Localization of the spectra of dual frames multipliers}} \\
	\vspace*{0.5cm}
	{\scshape{Rosario Corso}}
\end{center}

\normalsize 
\vspace*{1cm}	

\small 

\begin{minipage}{12.5cm}
	{\scshape Abstract.} This paper concerns dual frames multipliers, i.e. operators in Hilbert spaces consisting of analysis, multiplication and synthesis processes, where the analysis and the synthesis are made by two dual frames, respectively. The goal of the paper is to give some results about the localization of the spectra of dual frames multipliers, i.e. to individuate regions of the complex plane containing the spectra using some information about the frames and the symbols.
\end{minipage}

\vspace*{.5cm}

\begin{minipage}{12.5cm}
\end{minipage}

\vspace*{.5cm}

\begin{minipage}{12.5cm}
\end{minipage}

\normalsize

\section{Introduction}

Frame multipliers have been objects of several studies \cite{Balazs_inv_mult2,Corso_mult1,Balazs_inv_mult,Detail_mult,Riesz_mult,Dual_mult,Comm_mult} and applications in physics \cite{Gazeau}, signal processing (in particular, Gabor multipliers \cite{FeiNow,Matz} attracted interest as time-variant filters) and acoustics \cite{BHNS,Balazs_appl}. Details about applications are discussed also in the survey \cite{Balazs_surv}.  

Frame multipliers are part of the Bessel multipliers which were introduced in \cite{Balazs_basic_mult} and we are going to recall. A {\it Bessel sequence} of a separable Hilbert space $\H$ is a sequence $\varphi=\{\varphi_n\}_{n\in \N}$ of elements of $\H$ such that there exists $B_\varphi>0$ and 
$$
\sum_{n\in \N} |\pin{f}{\varphi_n}|^2\leq B_\varphi \|f\|^2, \qquad \forall f \in \H.
$$
The constant $B_\varphi$ is called a {\it Bessel bound} of $\varphi$. A sequence $\varphi=\{\varphi_n\}_{n\in \N}$ is a {\it frame} of $\H$ if there exist $A,B>0$ ({\it lower bound} and {\it upper bound} of $\varphi$, respectively) such that
\begin{equation}
\label{def_frame}
A \|f\|^2\leq \sum_{n\in \N} |\pin{f}{\varphi_n}|^2\leq B \|f\|^2, \qquad \forall f \in \H.
\end{equation}
Now let $\varphi=\{\varphi_n\}_{n\in \N},\psi=\{\psi_n\}_{n\in \N}$ be Bessel sequences of $\H$ and $m=\{m_n\}_{n\in \N}\in \ell^\infty$, i.e. a bounded complex sequence. The operator $M_{m,\varphi,\psi}$ given by 
$$M_{m,\varphi,\psi}f=\sum_{n=1}^\infty m_n\pin{f}{\psi_n} \varphi_n \qquad f\in \H,$$ 
is called the {\it Bessel multiplier} of $\varphi$, $\psi$ with {\it symbol}  $m$. Correspondent versions of Bessel multipliers have been studied also in continuous and distributional contexts 
(see \cite{ Mult_cont,Corso_distr_mult,TTT}). A Bessel multiplier $M_{m,\varphi,\psi}$ is called a {\it frame multiplier} if $\varphi$ and $\psi$ are frames.  

This paper deals with the spectra of {\it dual frames multipliers}, i.e. multipliers $M_{m,\varphi,\psi}$ where $\varphi$ and $\psi$ are dual frames and $m\in \ell^\infty$. Two frames $\varphi$ and $\psi$ of $\H$ are called {\it dual} if $f=\sum_{n\in \N} \pin{f}{\varphi_n}\psi_n$ for every $f\in \H$ (or, equivalently, $f=\sum_{n\in \N} \pin{f}{\psi_n}\varphi_n$ for every $f\in \H$). In particular, the study of this paper is inspired by the following result for Bessel multipliers, which is an immediate consequence of \cite[Theorem 6.1]{Balazs_basic_mult}, i.e. of the fact that 
\begin{equation}
\label{norm_Bessel}
\n{M_{m,\varphi,\psi}} \leq \sup_{n\in \N} |m_n| {B_\varphi}^\mez {B_\psi}^\mez,
\end{equation}
where $M_{m,\varphi,\psi}$ is any Bessel multiplier with $B_\varphi$ and $B_\psi$ some Bessel bounds of $\varphi$ and $\psi$, respectively.

\begin{pro}
	\label{spec_intro}
	The spectrum of any Bessel multiplier $M_{m,\varphi,\psi}$ is contained in the closed disk centered the origin with radius $\sup_{n\in \N} |m_n| {B_\varphi}^\mez {B_\psi}^\mez$, where $B_\varphi$ and $B_\psi$ are Bessel bounds of $\varphi$ and $\psi$, respectively. 
\end{pro}

Proposition \ref{spec_intro} gives an information about the location of the spectra of Bessel multipliers in the complex plane. However, the given estimate may be too large for the spectra of dual frames multipliers\footnote{Indeed, if $m\in \ell^\infty$, $\varphi$ and $\psi$ are dual Riesz bases (for the definition see the end part of Section \ref{sec:pre}), then the spectrum of $M_{m,\varphi,\psi}$ is the closure of the set $\{m_n\}_{n\in \N}$ (see \cite[Proposition 2.1]{Bag_Riesz} or \cite[Section 5.1]{Corso_seq}), which is in general smaller than the closed disk centered the origin with radius $\sup_{n\in \N} |m_n| {B_\varphi}^\mez {B_\psi}^\mez$.}. The main results of the paper, Theorems \ref{th_pert_Riesz} and \ref{th_pertIV}, provide more accurate localization results for the spectra of dual frames multipliers $M_{m,\varphi,\psi}$ under some conditions on $\varphi$ and $\psi$. We also stress that these conditions occur in many frames used in applications (see Remark \ref{Gabor}). 

A localization of the spectrum of $M_{m,\varphi,\psi}$ may show that $M_{m,\varphi,\psi}$ is invertible. The invertibility of multipliers was a subject faced in \cite{Balazs_inv_mult2,Corso_mult1,Balazs_inv_mult,Detail_mult,Riesz_mult,Dual_mult} and Theorems \ref{th_pert_Riesz} and \ref{th_pertIV} bring new results in this direction (see Remark \ref{rem_new}).  

Moreover, the knowledge of a region containing the spectrum of $M_{m,\varphi,\psi}$ gives, in particular, information about the distribution of the possible eigenvalues of $M_{m,\varphi,\psi}$. In connection with this subject, recently in \cite{Corso_mult1} some types of dual frames multipliers with at most countable spectra have been studied. 

The paper is organized as follows. In Section \ref{sec:pre} we recall some basic notions of frame theory, while we give some preliminary localization results about the spectra of dual frames multipliers in Section \ref{sec_bas}. Finally, Sections \ref{sec_case1} and \ref{sec_case2} contain the main results mentioned above together with examples.

\section{Preliminaries}
\label{sec:pre}

Throughout the paper $\H$ indicates a separable Hilbert space with inner product $\pint$ and norm $\nor$. 
Given an operator $T$ acting between two Hilbert space $\H_1$ and $\H_2$, we denote by  $R(T)$ and $N(T)$ the  {\it range} and {\it kernel} of $T$, respectively, and by $T^*$ its {\it adjoint} when $T$ is bounded. 

If $T:\H\to \H$ is a bounded operator, then we write $\rho(T)$ and $\sigma(T)$ for the {\it resolvent set} and {\it spectrum} of $T$, respectively. 
Throughout the paper we will apply the following standard perturbation result (for a reference see, for instance, Theorem IV.1.16 of \cite{Kato}). 

\begin{lem}
	\label{lem_pert}
	Let $T,B:\H\to \H$ be bounded operators. If $T$ is bijective and $\|B\|<\|T^{-1}\|^{-1}$, then $T+B$ is bijective. 
\end{lem}

We denote by $\ell^2$ (respectively, $\ell^\infty$)  the usual spaces of square summable (respectively, bounded) complex sequences indexed by $\N$. A {\it limit point} for $m\in \ell^\infty$ is the limit of a converging subsequence of $m$. 

In the introduction we gave the definitions of Bessel sequences and frames. Here, we recall some other notions and elementary results about frame theory \cite{Chris}. 
A sequence $\varphi=\{\varphi_n\}_{n\in \N}$ is {\it complete} in $\H$ if its linear span is dense in $\H$ if and only if $\pin{\varphi_n}{f}=0$ for every $n\in \N$ implies $f=0$.  A frame for $\H$ is, in particular, complete in $\H$.

Let $\varphi$ be a frame for $\H$. We say that $\varphi$ is a {\it Parseval frame} if \eqref{def_frame} holds with $A=B=1$.
The operator $S:\H \to \H$, defined by 
$$
S f =\sum_{n\in I} \pin{f}{\varphi_n}\varphi_n \qquad f\in \H,
$$
is well-defined, bounded, bijective and it called the {\it frame operator} of $\varphi$. The sequence  $\{S^{-1} \varphi_n\}_{n\in \N}$ is a dual frame of $\varphi$, called the {\it canonical dual}, and $\{S^{-\mez} \varphi_n\}_{n\in \N}$ is a Parseval frame for $\H$, called the {\it canonical Parseval frame} of $\varphi$.

A {\it Riesz basis} $\varphi$ for $\H$  is a complete sequence in $\H$ satisfying for some $A,B>0$ 	
\begin{equation}
\label{def_Riesz}
	A \sum_{n\in \N} |c_n|^2 \leq \left \| \sum_{n\in \N} c_n \varphi_n \right \|^2 \leq B\sum_{n\in \N} |c_n|^2, \qquad\forall \{c_n\}\in \ell^2.
\end{equation}

\no A Riesz basis $\varphi$ for $\H$ is a frame for $\H$, the constants in \eqref{def_frame} can be chosen as in \eqref{def_Riesz} and the canonical dual of $\varphi$ is a Riesz basis too (called {\it dual Riesz basis} of $\varphi$).

\section{Basic localization results}
\label{sec_bas}

In this section we give two preliminary localization results of the spectra of dual frames multipliers (Propositions \ref{num_range_M} and \ref{pertII}) without requiring specific properties of the two frames. For the first one we need the notion of numerical range. Given a bounded operator $T:\H\to \H$ the {\it numerical range} of $T$ is the set $\rn_T=\{\pin{Tf}{f}:f\in \H,  \n{f}=1\}$. We recall also that the spectrum of $T$ is contained in the closure of $\rn_T$ (see \cite[Corollary V.3.3]{Kato}). 

In addition, we are going to use the following lemma, which state that to examine the spectrum of a dual frame multiplier $M_{m,\varphi,\psi}$ where $\psi$ is, in particular, the canonical dual frame of $\varphi$, we can just consider a  multiplier determined by a Parseval frame.

\begin{lem}
	\label{spec_agg}
	Let $\varphi$ be a frame for $\H$, $\psi$ its canonical dual frame and $m\in \ell^\infty$. Then  $M_{m,\varphi,\psi}$ is similar to $M_{m,\rho,\rho}$, where $\rho$ is the Parseval frame associated to $\varphi$, and so $\sigma(M_{m,\varphi,\psi})=\sigma(M_{m,\rho,\rho})$. In particular, if $m$ is a real (resp., non-negative) sequence, then  $M_{m,\varphi,\psi}$ is similar to a self-adjoint  (resp., non-negative) operator and $\sigma(M_{m,\varphi,\psi})$ is real (resp., non-negative).  	
\end{lem}
\begin{proof}
	Let $S$ be the frame operator of $\varphi$, which we recall is a bijective operator. It is immediate to see that $S^{-\mez}M_{m,\varphi,\psi}S^{\mez}=M_{m,\rho,\rho}$, where $\rho=S^{-\mez}\varphi$ is the Parseval frame associated to $\varphi$. The rest of the statement follows easily.  
\end{proof}

\begin{pro}
	\label{num_range_M}
	Let $\varphi$ be a frame for $\H$, $\psi$ its canonical dual and $m\in \ell^\infty$. Then $\sigma(M_{m,\varphi,\psi})$ is contained in the closed convex hull of $m$, i.e. the closure of the set $\{\sum_{n\in \N}a_n m_n: \sum_{n\in \N}|a_n|^2=1\}$.
\end{pro}
\begin{proof}
	Since the spectrum is preserved by similarity, we can confine to the case where $\varphi=\psi$ is a Parseval frame by Lemma \ref{spec_agg}. We note that 
	$$\pin{M_{m,\varphi,\varphi} f}{f}=\sum_{n\in \N} m_n|\pin{f}{\varphi_n}|^2,$$ therefore, because $\n{f}^2=\sum_{n\in \N} |\pin{f}{\varphi_n}|^2$, the numerical range (and then also the spectrum) of $M_{m,\varphi,\varphi}$ is contained in the closed convex hull of $m$. 
\end{proof}

\begin{rem}
	\begin{enumerate}
		\item[(i)] Under the hypothesis of Proposition \ref{num_range_M}, if in addition $m$ is a real sequence, we have that $\sigma(M_{m,\varphi,\psi})\subset [\inf_{n\in \N} m_n , \sup_{n\in \N}m_n ]$. 
		\item[(ii)] If $\varphi$ is not a Parseval frame and $\psi$ is its canonical dual, then the numerical range of $M_{m,\varphi,\psi}$ is not necessarily contained in the closed convex hull of $m$ (even though $\sigma(M_{m,\varphi,\psi})$ is). This follows from the fact that the numerical range is not invariant under similarity. 
	\end{enumerate}
\end{rem}

The statement of Proposition \ref{num_range_M} may not hold if $\psi$ is just a dual frame of $\varphi$. For example, take $\varphi=\{e_1,e_1,e_2,\dots, e_n, \dots\}$, $\psi=\{ie_1,(1-i)e_1,e_2,\dots, e_n, \dots\}$, where $\{e_n\}_{n\in \N}$ is an orthonormal basis of $\H$ and $m=\{2,1,1,\dots\}$.  For generic dual frames we can actually state the following.

\begin{pro}
	\label{pertII}
	Let $\varphi,\psi$ be dual frames for $\H$ with upper bounds $B_\varphi,B_\psi$, respectively. Let $m\in \ell^\infty$. If $\lambda,\mu\in \C$ and 
	\begin{equation}
	\label{cond_pertII}
	\|m-\mu\|_\infty B_\varphi^\mez B_\psi^\mez<|\mu-\lambda|,
	\end{equation}
	then $\lambda \in \rho(M_{m,\varphi,\psi})$. In particular,
	\begin{enumerate}
		\item if $m$ is contained in the disk of center $\mu$ with radius $r$, then $\sigma(M_{m,\varphi,\psi})$ is contained in the disk of center $\mu$ with radius $rB_\varphi^\mez B_\psi^\mez$;
		\item if $m$ is real, then $\sigma(M_{m,\varphi,\psi})$ is contained in the disk of center $\frac 12 (\sup_{n\in \N} m_n +\inf_{n\in \N} m_m)$ with radius $\frac 12 (\sup_{n\in \N} m_n -\inf_{n\in \N} m_m)B_\varphi^\mez B_\psi^\mez$.
	\end{enumerate}
\end{pro}
\begin{proof}
	If \eqref{cond_pertII} holds, then by \eqref{norm_Bessel} we have  $\|M_{m-\mu,\varphi,\psi}\|<|\mu-\lambda|$ and by Lemma \ref{lem_pert} we have  $\lambda-\mu\in \rho(M_{m-\mu,\varphi,\psi})$, i.e.  $\lambda\in \rho(M_{m,\varphi,\psi})$. Now, the rest of the statement is immediate.	
\end{proof}

When $\psi$ is the canonical dual of $\varphi$, Proposition \ref{num_range_M} gives a more accurate result than Proposition \ref{pertII}. 

\section{Main result 1}
\label{sec_case1}

For the localization result in this section we make the assumption that a frame contains a Riesz basis. Note that this is not a very strong requirement. Indeed, it is satisfied by many frames used in applications 
(see Remark \ref{Gabor} below for a consideration about Gabor and wavelet frames).

\begin{theo}
	\label{th_pert_Riesz}
	Let $\varphi,\psi$ be frames for $\H$ such that for some $I\subset \N$, $\{\varphi_{n}:n\in I\}$, $\{\psi_{n}:n\in I\}$ are Riesz bases for $\H$ with lower frame bounds $A_{\varphi,1}$ and $A_{\psi,1}$, respectively. Moreover, let $B_{\varphi,2}$ and $B_{\psi,2}$ be Bessel bounds of $\{\varphi_{n}:n\in \N\backslash I\}$ and  $\{\psi_{n}:n\in \N\backslash I\}$, respectively. Let $m\in \ell^\infty$. If 
	\begin{equation}
	\label{pert}
	\sup_{n\in \N\backslash I}{|m_{n}|}{B_{\varphi,2}}^\mez {B_{\psi,2}}^\mez< \inf_{n\in I}{|m_{n}|}{A_{\varphi,1}}^\mez {A_{\psi,1}}^\mez,
	\end{equation}
	then $M_{m,\varphi,\psi}$ is bijective. \\
	If, in addition, $\varphi$ and $\psi$ are dual frames and
	\begin{equation}
	\label{pert2}
	\sup_{n\in \N \backslash I}{|m_{n}-\lambda |}{B_{\varphi,2}}^\mez {B_{\psi,2}}^\mez< \inf_{n\in I}{|m_{n}-\lambda |}{A_{\varphi,1}}^\mez {A_{\psi,1}}^\mez,
	\end{equation}
	then $\lambda \in \rho(M_{m,\varphi,\psi})$. 
\end{theo}
\begin{proof}
	We can write $M_{m,\varphi,\psi}=M_1+M_2$, where  $M_1=M_{m^{(1)},\varphi^{(1)},\psi^{(1)}}$ and $M_2=M_{m^{(2)},\varphi^{(2)},\psi^{(2)}}$, $m^{(1)}=\{m_n: {n\in I}\}$, $m^{(2)}=\{m_n: {n\in \N \backslash I}\}$, $\varphi^{(1)}=\{\varphi_n: {n\in I}\}$, $\varphi^{(2)}=\{\varphi_n: {n\in \N \backslash I}\}$, $\psi^{(1)}=\{\psi_n: {n\in I}\}$, $\psi^{(2)}=\{\psi_n: {n\in \N \backslash I}\}$. First of all,  $\inf_{n\in I}{|m_{n}|}>0$ holds by \eqref{pert}, so $M_1$ is bijective by 
	\cite[Theorem 5.1]{Balazs_inv_mult}. Moreover, \eqref{pert} allows to apply Lemma \ref{lem_pert} because   $$\n{M_2}\leq \sup_{n\in\N \backslash I}{|m_{n} |}{B_{\varphi,2}}^\mez {B_{\psi,2}}^\mez$$ by \eqref{norm_Bessel}, and $$\inf_{n\in I}{|m_{n} |}{A_{\varphi,1}}^\mez {A_{\psi,1}}^\mez\leq \n{M_1^{-1}}^{-1}$$ 
	by Propositions 7.7 and 7.2 of \cite{Balazs_basic_mult} and the fact that a Bessel bound of the canonical dual of $\varphi$ (resp., $\psi$) is ${A_{\varphi,1}}^{-1}$ (resp., ${A_{\psi,1}}^{-1}$). 
	The second part of the statement now follows from the fact that $M_{m,\varphi,\psi}-\lambda I=M_{m-\lambda,\varphi,\psi}$ when $\varphi$ and $\psi$ are dual frames. 
\end{proof}

We show an application of Theorem \ref{th_pert_Riesz} with an example of multiplier with $0-1$ symbol (i.e. a sequence made only of $0$ and $1$)\footnote{Such a multiplier often occurs in applications, see e.g. \cite{Balazs_surv}.}.

\begin{exm}
	\label{exm_1}
	Let $\varphi$ be a Parseval frame for $\H$ such that $\{\varphi_{2n}\}_{n\in \N}$ is a Riesz basis for $\H$ with lower bound $A$. Clearly, we have $0<A<1$. Consequently, $\{\varphi_{2n-1}\}$ is a Bessel sequence with bound $1-A$. \\	
	Let, moreover, $m$ be a sequence of $0$ and $1$.  With these choices, we apply Theorem \ref{th_pert_Riesz} to $M_{m,\varphi,\varphi}$. Condition \eqref{pert2} is 
	\begin{equation}
	\label{pert_exm1}
	\sup_{n\in \N} |m_{2n-1}-\lambda|(1-A)<\inf_{n\in \N} |m_{2n}-\lambda|A.
	\end{equation}
	We have
	\begin{equation*}
	\inf_{n\in \N} \{|-\lambda|,|1-\lambda|\}=
	\begin{cases}
	-\lambda \quad\quad\quad \lambda < 0,\\
	\lambda \hspace*{1.45 cm} 0\leq \lambda \leq \frac 1 2 ,\\
	1-\lambda \qquad \frac 12 <\lambda \leq 1,\\
	\lambda-1 \qquad 1<\lambda,
	\end{cases}
	\end{equation*}
	and 
	\begin{equation*}
	\sup_{n\in \N} \{|-\lambda|,|1-\lambda|\}=
	\begin{cases}
	1-\lambda \qquad\lambda < 0,\\
	1-\lambda \hspace*{0.8 cm} 0\leq \lambda \leq \frac 1 2 ,\\
	\lambda \hspace*{1.4 cm} \frac 12 <\lambda \leq 1,\\
	\lambda \hspace*{1.4 cm} 1<\lambda.\\
	\end{cases}
	\end{equation*}
	
	Since, by Proposition \ref{num_range_M}, we already know that $\sigma(M_{m,\varphi,\varphi})\subseteq [0,1]$, we need to check the validity of \eqref{pert_exm1} only for $0\leq\lambda\leq 1$. We note that if $0\leq \lambda \leq \frac 1 2$, then \eqref{pert_exm1} is true if and only if $\lambda>1-A$, which makes sense only if $A> \frac 12 $. On the other hand, if $\frac 12 <\lambda \leq 1$, then \eqref{pert_exm1} is true if and only if $\lambda<A$, which makes sense again only if $A> \frac 12 $. Thus, by Theorem \ref{th_pert_Riesz}, we can write  that if $A> \frac 12$ 
	$$\sigma(M_{m,\varphi,\varphi})\subseteq [0,1-A]\cup [A,1].$$	
\end{exm}

As particular case of Theorem \ref{th_pert_Riesz} we get the following. 

\begin{cor}
	Let $\varphi$ be a frame for $\H$ with bounds $A$ and $B$ such that for some $I\subset \N$ $\{\varphi_{n}:n\in I\}$ is a Riesz basis for $\H$ with lower frame bound $A'$. Let $\psi$ be the canonical dual of $\varphi$ and $m\in\ell^\infty$. If 
	\begin{equation*}
	\label{pert2'}
	\sup_{n\in \N \backslash I}{|m_{n}-\lambda |}\frac{B-A'}{A}< \inf_{n\in I}{|m_{n}-\lambda |}\frac{A'}{B},
	\end{equation*}
	then $\lambda \in \rho(M_{m,\varphi,\psi})$. 
\end{cor}
\begin{proof}
	The statement follows by Theorem \ref{th_pert_Riesz} once noticed that $\{\psi_n\}_{n\in I}=\{S^{-1}\varphi_n\}_{n\in I}$ is a Riesz basis with lower bound $\frac{A'}{B^2}$, $\{\psi_n\}_{n\in \N\backslash I}$ has Bessel bound $B-A'$ and $\{\psi_n\}_{n\in\N\backslash I}=\{S^{-1}\varphi_n\}_{n\in \N\backslash I}$ has Bessel bound $\frac{B-A'}{A^2}$.
\end{proof}

\begin{rem}
	\label{Gabor}
	Gabor and wavelet frames are classical frames which occur in applications  (see \cite{Chris,Daubechies,Groechenig_b}). A (regular) Gabor frame for $L^2(\R)$ is a frame of the form 
	$$
	\mathcal{G}(g,a,b)=\{E_{b}^mT_{a}^ng\}_{m,n\in \Z}
	$$
	where $g\in L^2(\R)$, $a,b>0$, $
	(T_af)(x)=f(x-a)$ and $(E_b f)(x)=e^{2 \pi i b x}f(x)$ for $x \in \R$. A Gabor frame which is a finite union of Riesz bases can be easily constructed in this way. Let $N\in \N$ and $\mathcal{G}(g,a,b)$ a Riesz basis for $L^2(\R)$. 
	A simple calculation shows that  $\mathcal{G}(g,\frac a N,b)$ (as well as $\mathcal{G}(g,a,\frac b N)$) is a frame for $L^2(\R)$ which is a union of $N$ Riesz bases.  
	
	Frames which are unions of Riesz bases can be found also in the context of wavelet frames. In particular, the frame multiresolution analysis technique (see \cite[Ch. 17]{Chris}) gives a way to construct wavelet frames which are unions of Riesz bases. 
\end{rem}

\section{Main result 2}
\label{sec_case2}

In this section we consider Parseval frames $\varphi$ for $\H$ which are unions of multiples of orthonormal bases. In other words, we can think that there exists $k\in \N$ such that
\begin{equation}
\label{frame_locII}
\{\varphi_{(i-1)k+j}:i\in \N\}=\{\alpha_j e_i^j:i\in \N\},
\end{equation}
where $\alpha_j\in \C\backslash\{0\}$ and $\{e_i^j:i\in \N\}$ is orthonormal basis for $\H$ for $j=1,\dots, k$. Also here, we remark that this condition occurs for frames used in application.  For instance, following Remark \ref{Gabor}, if $\mathcal{G}(g,a,b)$ is an orthonormal basis for $L^2(\R)$, then $\frac 1 N\mathcal{G}(g,\frac a N,b)$ and $\frac 1 N\mathcal{G}(g, a ,\frac b N)$ are Parseval frames and unions of $N$ multiples of orthonormal bases.

\begin{theo}
	\label{th_pertIV}
	Let $\varphi$ be as in \eqref{frame_locII}, $m\in \ell^\infty$ and $l_1,\dots,l_k\in \C$. If $\lambda \in \C$ and
	\begin{equation}
	\label{pert_III}
	\sum_{j=1}^k |\alpha_j|^2 \sup_{i\in \N} |m_{(i-1)k+j} -l_j| <\left |\sum_{j=1}^k |\alpha_j|^2 l_j -\lambda \right |,
	\end{equation}
	then $\lambda \in \rho(M_{m,\varphi,\varphi})$.  As a  consequence, if $m$ is a real sequence, then  
	\begin{equation}
	\label{pert_IIIfin}
	\sigma(M_{m,\varphi,\varphi})\subseteq \left [\sum_{j=1}^k |\alpha_j|^2 \inf_{i\in \N} m_{(i-1)k+j},\sum_{j=1}^k |\alpha_j|^2 \sup_{i\in \N} m_{(i-1)k+j}\right ].
	\end{equation}
\end{theo}
\begin{proof}
	First of all, we note that $\lambda \neq \sum_{j=1}^k |\alpha_j|^2 l_j$ by \eqref{pert_III}. 
	We have
	$M_{m,\varphi,\varphi}f=\sum_{j=1}^k |\alpha_j|^2 l_j f  
	+M_{m',\varphi,\varphi}$ where $m'=\{m'_n\}$ and $m'_{(i-1)k+j}=m_{(i-1)k+j}-l_j$ for $i\in \N$ and $j=1,\dots,k$. Thus the first statement follows by Lemma \ref{lem_pert} noting that  $$\n{M_{m',\varphi,\varphi}}\leq \sum_{j=1}^k |\alpha_j|^2 \sup_{i\in \N} |m_{(i-1)k+j} -l_j|.$$ 
	Now assume that $m$ is real, i.e. $M_{m,\varphi,\varphi}$ is self-adjoint. Therefore \eqref{pert_III} implies that 
	\begin{equation}
	\label{1}
	\sigma(M_{m,\varphi,\varphi})\subseteq \left [\sum_{j=1}^k |\alpha_j|^2(l_j -\sup_{i\in \N} |m_{(i-1)k+j} -l_j|),\sum_{j=1}^k |\alpha_j|^2(l_j +\sup_{i\in \N} |m_{(i-1)k+j} -l_j|)\right ]
	\end{equation}
	Choosing in \eqref{1}, first $l_j<\inf_{i\in \N} m_{(i-1)k+j}$  and then $l_j>\sup_{i\in \N} m_{(i-1)k+j}$ for every $j=1,\dots, k$,  we find \eqref{pert_IIIfin}.
\end{proof}

\begin{exm}
	\label{exm_2}
	Let $\varphi=\{\frac{1}{\sqrt{2}}e_1,\frac{1}{\sqrt{2}}f_1, \frac{1}{\sqrt{2}}e_2, \frac{1}{\sqrt{2}}f_2, \dots\}$ where $\{e_n\}$ and $\{f_n\}$ are orthonormal bases for $\H$. Furthermore, let $m=\{m_n\}$ be such that $m_{4n-3}=0$,  $m_{4n-2}=\frac 13$, $m_{4n-1}=\frac 23$ and  $m_{4n}=1$,  $n\in \N$. 
	Taking into account Proposition \ref{num_range_M}, the spectrum of $M_{m,\varphi,\varphi}$ is contained in $[0,1]$. 
	This estimate can be improved by Theorem \ref{th_pertIV}: in particular, we obtain  that  $\sigma(M_{m,\varphi,\varphi})\subseteq[\frac 1 6,\frac 5 6] $. 	
\end{exm}

\begin{rem}
	Theorem \ref{th_pertIV} is not a special case of Theorem  \ref{th_pert_Riesz} (and vice-versa). In particular, Theorem \ref{th_pert_Riesz} gives no improvement on the localization of the spectrum in Example \ref{exm_2}. On the other hand,  Theorem \ref{th_pertIV} does not add any further information about the spectrum of the multiplier in Example \ref{exm_1}, even in the case where $\{\varphi_{2n}\}_{n\in \N}$ and $\{\varphi_{2n+1}\}_{n\in \N}$ are multiples of orthonormal bases.
\end{rem}

\begin{rem}
	\label{rem_new}
	Theorems \ref{th_pert_Riesz} and \ref{th_pertIV} give, in particular, new criteria of invertibility in comparison to the results in \cite{Balazs_inv_mult}. For instance, let  $\varphi=\{\frac{1}{\sqrt{2}}e_1,\frac{1}{\sqrt{2}}f_1, \frac{1}{\sqrt{2}}e_2, \frac{1}{\sqrt{2}}f_2, \dots\}$ where $\{e_n\}$ and $\{f_n\}$ are orthonormal bases for $\H$ and $m=\{m_n\}$ is such that $m_{2n-1}=\frac{1}{n+1}$ and $m_{2n}=2-\frac{1}{n+1}$, $n\geq 1$. Both Theorems \ref{th_pert_Riesz} and \ref{th_pertIV} show that 
	$\sigma(M_{m,\varphi,\varphi})\subseteq [\frac 3 4,\frac 54 ]$. In particular,
	$M_{m,\varphi,\varphi}$ is invertible. However, Propositions 4.1, 4.2 and 4.4 of \cite{Balazs_inv_mult} do not apply to this multiplier $M_{m,\varphi,\varphi}$. 
\end{rem}

\vspace*{0.5cm}
\begin{center}
\textsc{Rosario Corso, Dipartimento di Matematica e Informatica} \\
\textsc{Università degli Studi di Palermo, I-90123 Palermo, Italy} \\
{\it E-mail address}: {\bf rosario.corso02@unipa.it}
\end{center}

\end{document}